\documentclass[12pt,a4paper]{article}
\usepackage[a4paper]{geometry}
\usepackage[T1]{fontenc}
\usepackage[utf8]{inputenc}
\usepackage[english]{babel}
\usepackage{lmodern}
\usepackage{amsmath,amsfonts,amssymb}
\usepackage{amsthm}
\usepackage{authblk}
\usepackage{blindtext}
\usepackage{hyperref}

\theoremstyle{definition}
\newtheorem{theorem}{Theorem}
\newtheorem{Def}{Definition}
\newtheorem{prop}{Proposition}
\newtheorem{cor}{Corollary}
\newtheorem{lem}{Lemma}
\newtheorem{rem}{Remark}
\newtheorem{ex}{Example}

\newcommand{\email}[1]{\href{mailto:#1}{\nolinkurl{#1}}}
\newcommand{\keywords}[1]{\textit{Keywords:} \ #1}
\newcommand{\subjclass}[2]{2010 \textit{Mathematics Subject Classification:} \ #1}

\title{Discrete correlation of order 2 of generalized Rudin--Shapiro sequences on alphabets of arbitrary size}
\author{Pierre-Adrien Tahay \\ email: \href{mailto:pierre-adrien.tahay@univ-lorraine.fr}{pierre-adrien.tahay@univ-lorraine.fr}}
\affil{Universit\'e de Lorraine, IECL, F-54000 Nancy, France}

\begin{document}
\date{}
\maketitle

\setlength\parindent{0em}



\begin{abstract}

In 2009, Grant, Shallit, and Stoll~\cite{grant_bounds_2009} constructed a large family of pseudorandom sequences, called generalized Rudin--Shapiro sequences, for which they established some results about the average of discrete correlation coefficients of order 2 in cases where the size of the alphabet is a prime number or a squarefree product of primes. We establish similar results for an even larger family of pseudorandom sequences, constructed via difference matrices, in the case of an alphabet of any size. The constructions generalize those from~\cite{grant_bounds_2009}. In the case where the size of the alphabet is squarefree and where there are at least two prime factors, we obtain an improvement in the error term by comparison with the result of Grant et al.~\cite{grant_bounds_2009}. 
\end{abstract}

\section{Introduction}

In 1997 and 1998, Mauduit and Sárközy published two papers~\cite{mauduit_finite_1997,mauduit_finite_1998} about pseudorandom sequences, i.e., deterministic sequences on finite alphabets sharing similar properties with random sequences. Various results, in particular the pseudorandomness of the Legendre symbol and the correlation of Champernowne, Thue--Morse and Rudin--Shapiro (or Golay--Rudin--Shapiro) sequences have been established. There exists a large literature on the subject. We refer to the recent papers~\cite{KatzSequencesLowCorrelation2018,KoniecznyGOWERSNORMSTHUEMORSE,MauduitPrimenumbersRudin2015,MauduitRudinShapirosequences2018,MeraimeasurespseudorandomnessNIST2018,QueffelecQuestionsThueMorseSequence2018} and their bibliographic references. In the same way as Grant et al.~\cite{grant_bounds_2009}, our work concerns the explicit construction of sequences with good discrete correlation properties. We extend their construction to get similar correlation properties of suitably generalized Rudin--Shapiro sequences valid for all alphabets. In the case where the size of the alphabet is a power of a prime, the error term obtained is the same for all powers and the same as the one of Grant et al.~\cite{grant_bounds_2009} when the power is equal to 1. When the size is a product of several powers of prime numbers, the error term is also independent of the powers chosen, but in the case where all the exponents in the powers are equal to 1, we obtain an improvement of the error term with respect to the result of Grant et al.~\cite{grant_bounds_2009}. Moreover, with our construction it is possible to recover the one of Grant et al.~\cite{grant_bounds_2009}.

\section{Definitions and state-of-the-art}

Throughout the paper, we use $\mathbb{Z}_p=\mathbb{Z}/p\mathbb{Z}$, $\mathbb{Z}^k_p=\underbrace{\mathbb{Z}/p\mathbb{Z}\times\cdots\times\mathbb{Z}/p\mathbb{Z}}_{k}$, and $\text{e}(x)=e^{2i\pi x}$ for all $x \in \mathbb{R}$. We make use of the usual Landau notation $O()$ for the error terms. We may use indices to indicate the dependence of the implied constant (such as $O_{k}()$ for a possible dependence on $k$). We also make use of the classical Vinogradov notation $\ll$.

\begin{Def}\label{def1}
Let $k \geq 2$ be an integer and let $x=x_0x_1\cdots$ be an infinite word on the alphabet $\left\{0,1,\ldots,k-1\right\}$. For a vector $(i,j)$ satisfying $0\leq i<j$ we define the \emph{discrete correlation coefficient} $\delta(i,j)$ of order $2$ by
$$\delta(i,j)=\begin{cases}
0,&\text{if } x_{i}=x_{j},\\
1,&\text{else.}
\end{cases}$$

Moreover, we define $C_{\textbf{r}}$ for all $\textbf{r}=(r_1,r_2)$ with $0\leq r_1<r_2$ by
$$C_{\textbf{r}}=\liminf\limits_{N\rightarrow\infty}\cfrac{1}{N}\displaystyle\sum_{n<N}\delta(n+r_1,n+r_2).$$
\end{Def}

The quantity $C_{\textbf{r}}$ measures in some sense how far a particular sequence is ``pseudorandom'' (see remark below). We allow $\textbf{r}$ to depend on $n$ in order to provide constructions of sequences that are robust. Note that this generalizes vastly the case when one fixes $\textbf{r}=(r_1,r_2)$ as a constant vector. Let us begin with a remark.

\begin{rem}\label{rem1} 
For a random sequence where every letter is picked independently with probability $1/k$ we have $C_{\textbf{r}}=1-1/k$ with probability $1$.
\end{rem}

The aim of this paper is to construct a large class of deterministic sequences over an alphabet that generalize the Rudin--Shapiro sequence. Let us begin by reminding its definition.

\begin{Def}[\cite{AlloucheAutomaticSequencesTheory2003} p.78] The Rudin--Shapiro (or Golay--Rudin--Shapiro) sequence $(a_n)_{n \geqslant 0} = 0,0,0,1,0,0,1,0, \ldots$ is defined for all $n \in \mathbb{N}$ by

\begin{center}
$a_n=$ (number of blocks ``11'' in the binary representation of $n$) mod 2.
\end{center}
\end{Def}

\begin{rem}[\cite{AlloucheAutomaticSequencesTheory2003} p.79] It is easy to prove the following equivalent definition: 

\begin{center}
$a_{2n}=a_n$ and $a_{2n+1}=\begin{cases}
(a_n+1) \ \text{mod} \ 2 &\text{if} \ n \equiv 1 \ (\text{mod} \ 2),\\
a_n &\text{if} \ n \equiv 0 \ (\text{mod} \ 2).
\end{cases}$
\end{center}

Thus, the Rudin--Shapiro sequence can be defined as follows:

$$a_0=0 \ \text{and} \ a_{2n+j}=(a_n+g(j,n)) \  \text{mod} \ 2$$ 
with $g(j,n)=\begin{cases}
1,&\text{if} \ j=1, \ n \equiv 1 \ (\text{mod} \ 2),\\
0,&\text{else.}
\end{cases}$
\end{rem}

From this observation, Grant, Shallit, and Stoll~\cite{grant_bounds_2009} suggested a definition of generalized Rudin--Shapiro sequences.

\begin{Def}\label{def3}
Let 
\begin{align*}
g : \left\{0,1,\ldots,k-1\right\} \times \mathbb{Z} &\longrightarrow \mathbb{Z}\\ 
(j,n) &\longmapsto g(j,n)
\end{align*} 
be such that for each $j$, the function $n \mapsto g(j,n)$ is periodic with period $k$. Moreover, let $g$ be such that for all integers $u,i \in \mathbb{N}$ with $0\leq u<u+i\leq k-1$ we have
\begin{center}
$\left\{(g(u+i,n)-g(u,n)) \ \text{mod} \ k: 0 \leq n \leq k-1 \right\}=\left\{0,1,\ldots,k-1\right\}$.
\end{center}

We call a sequence $(\hat{a}(n))_{n \geq 0}$ over the alphabet $\left\{0,1,\ldots,k-1\right\}$ a \emph{generalized Rudin--Shapiro sequence} if there exists a sequence of integers $(a(n))_{n \geq 0}$ such that $\hat{a}(n) \equiv a(n)$ mod $k$ and
\begin{center}
$a(nk+j)=a(n)+g(j,n) \quad \text{for} \ 0\leq j \leq k-1, \ n\geq 1$.
\end{center}
\end{Def}

\begin{rem}
In order to define completely the sequence, we can fix (arbitrarily) the first values $a(0),\ldots,a(k-1)$ and the others are obtained recursively by the last relation.
\end{rem}

\begin{rem}
Allouche and Bousquet-Mélou~\cite{AlloucheFacteurssuitesRudinShapiro1994} studied in detail a generalization of the Rudin--Shapiro sequence within the framework of binary alphabets and paperfolding sequences. Rider~\cite{RiderTransformationsFouriercoefficients} defined a first generalization of the Rudin--Shapiro sequence over alphabets such that the size is a prime number, and M.~Queffélec~\cite{Queffelec_suites_1987} extended the definition for alphabets of arbitrary size and studied its spectral measure. In the definition introduced by Grant et al., these sequences correspond to the special case when the size of the alphabet is a prime number and the function $g$ is defined by $g(j,n)=jn \ \text{mod} \ k$ (see Example~\ref{ex2}). Allouche and Liardet~\cite{allouche_generalized_1991} also extended Queffélec's construction and proved that their sequences, as the classical Rudin--Shapiro sequence, still have the Lebesgue measure as spectral measure. In this paper, we do not look at spectral measure properties, but only at properties about discrete correlation of order 2, taking up the same point of view as Grant et al.~\cite{grant_bounds_2009}. 
\end{rem}

The two main results of Grant, Shallit, and Stoll~\cite{grant_bounds_2009} are as follows.

\begin{theorem}[Theorem 3.1 of \cite{grant_bounds_2009}]\label{theo1} 
Let $(\hat{a}(n))_{n \geq 0}$ be a generalized Rudin--Shapiro sequence over $\left\{0,1,\ldots,k-1\right\}$ with $k$ prime. Moreover, let $0 \leq r_1 < r_2$. Then, as $N \rightarrow \infty$, we have

$$\displaystyle\sum_{n<N}\delta(n+r_1,n+r_2)=N\left(1-\frac{1}{k}\right)+O_k\left((r_2-r_1)\text{log}\frac{N}{r_2-r_1}+r_2\right).$$
\end{theorem}

We note that the main term lines up exactly with the probabilistic one.

With this result, one can also prove that the main term is asymptotically larger than the error term as long as $r_2=o(N)$ (Corollary 3.2 of~\cite{grant_bounds_2009}).

Now, using a bijection between $\mathbb{Z}_{p_1} \times \cdots \times \mathbb{Z}_{p_d}$ and $\mathbb{Z}_{p_1 \cdots p_d}$, it is possible to construct a sequence over an alphabet whose size is squarefree and obtain similar properties about the correlation of order 2 of the sequence.
 
\begin{theorem}[Theorem 3.3 of \cite{grant_bounds_2009}]\label{theo2}
Let $d \geq 2$ and let $k=p_1\cdots p_d$ be a product of pairwise distinct primes. Let $c_1=1$ and $c_i=p_1\cdots p_{i-1}$ for $2 \leq i \leq d$. We define the sequence $(\hat{a}(n))_{n \geq 0}$ by
\begin{center}
$\hat{a}(n) \equiv a(n) \ \text{mod} \ k$,
\end{center}
where $(a(n))_{n \geq 0}$ is defined by $a(n)=c_1a_1(n)+\cdots+c_d a_d(n)$ and $(a_i(n))_{n \geq 0}$ satisfies the recursive relation 

$$a_i(p_in+j)=a_i(n)+g_i(j,n), \quad 1 \leq i \leq d,$$
for $n \geq 1$ and $0 \leq j \leq p_i-1$ and where the $g_i$ are functions which satisfy the conditions of Definition~\ref{def3}. Moreover, let $0 \leq r_1 < r_2$ and $0 < \gamma < 1$. Then, as $N \rightarrow \infty$ we have,

\begin{align*}
&\displaystyle\sum_{n<N}\delta(n+r_1,n+r_2)\\
&=N\left(1-\frac{1}{k}\right)+O_k\left((r_2-r_1)N^{1-\frac{\gamma}{d}}+(r_2-r_1)N^{1-\gamma}\text{log}\frac{N^{\frac{\gamma}{d}}}{r_2-r_1}+N^{\gamma}+r_2\right).
\end{align*}
\end{theorem}

Similarly, with this result, one can also prove that the main term is asymptotically larger than the error term as long as $r_2=o(N^{\frac{1}{d}})$ (Corollary 3.4 of~\cite{grant_bounds_2009}).

\begin{rem}
The previous construction cannot be used for an alphabet whose size is not squarefree because the proof of Theorem~\ref{theo2} requires the result of Theorem~\ref{theo1} that is only valid for a prime number and not for a power of a prime number. To overcome this obstacle, we use new constructions obtained via difference matrices.
We develop this crucial point in the following section, in order to generalize these two results to an alphabet of arbitrary size.
\end{rem}

The rest of the paper is structured as follows. In Section~\ref{Section3} we introduce difference matrices and give several examples. In
Section~\ref{Section4} we present our two main results (Theorem~\ref{theo4} and Theorem~\ref{theo5}), in Section~\ref{Section5} we give their proofs and we end the paper with some open questions in Section~\ref{Section6}.

\section{Difference matrices}\label{Section3}

Difference matrices play a central role in our constructions to generalize the previous results. We refer to~\cite{hedayat_orthogonal_1999} and~\cite{lampio_classification_nodate} for an overview on difference matrices. We here give an introduction to the theory of this kind of matrices with some examples. We exchange the role of the rows and the columns in comparison with~\cite{hedayat_orthogonal_1999} and~\cite{lampio_classification_nodate}.

\begin{Def}[\cite{hedayat_orthogonal_1999,lampio_classification_nodate}] Let $(G,+)$ be a finite abelian group of order $s$. A \emph{difference matrix} $D=(d_{ij})$ of size $r\times c$ with entries in $G$, is a matrix such that for all $i$ and $j$ with $1\leq i,j \leq c$, $i \neq j$, the set
\begin{center}
$\left\{d_{li}-{d_{lj}} : 1\leq l \leq r\right\}$
\end{center}
contains every element of $G$ equally often.
\end{Def}

\begin{ex}
$\begin{pmatrix}
0 & 0 & 0 \\
0 & 1 & 2 \\
0 & 2 & 1
\end{pmatrix}$
is a difference matrix over $\mathbb{Z}_3$.
\end{ex}

We let $D(r,c,G)$ denote the set of all difference matrices of size $r \times c$ with entries in the group $G$.

\begin{ex}[Example 6.3 of \cite{hedayat_orthogonal_1999}]\label{ex2} Let $k$ be a prime number. Then, the square matrix $A=(a_{ij})$ of size $k \times k$ defined by $a_{ij}=ij$ mod $k$ for all $1 \leq i,j \leq k$ is a matrix in $D(k,k,\mathbb{Z}_k)$.
\end{ex}

This result ensures that it is possible to build explicitly an example of function $g$ in the sense of Definition~\ref{def3} when the size of the alphabet is a prime number. Every set $\left\{(g(u+i,n)-g(u,n)) \ \text{mod} \ k: 0 \leq n \leq k-1 \right\}$ with $u$ and $i$ integers such that $0 \leq u < u+i \leq k-1$ are equivalent to a difference between two distinct columns. Consequently, Theorem~\ref{theo1} concerns a non-empty class of generalized Rudin--Shapiro sequences (see also Example 1 of~\cite{grant_bounds_2009}).

Ge~\cite[Lemma 3.1]{ge_g4;1-difference_2005} showed by elementary means that for an even integer $k\geq 4$, the set $D(k,k,\mathbb{Z}_k)$ is empty. In particular, the set $D(4,4,\mathbb{Z}_4)$ is empty. In other words, there is no square difference matrix of size $4$ over $\mathbb{Z}_4$. However, the set $D(4,4,\mathbb{Z}_2 \times \mathbb{Z}_2)$ is non-empty. Indeed, it is easy to check that the matrix
\begin{equation}
\label{matrix}
M=\begin{pmatrix}
(0,0) & (0,0) & (0,0) & (0,0) \\
(0,0) & (0,1) & (1,0) & (1,1) \\
(0,0) & (1,0) & (1,1) & (0,1) \\
(0,0) & (1,1) & (0,1) & (1,0) 
\end{pmatrix}
\end{equation}
is an element of this set, see~\cite[p.22]{hedayat_orthogonal_1999}.

More generally, we have the following result. For the sake of completeness we give below an explicit proof.
\begin{prop}\label{MatDif}
(\cite{hedayat_orthogonal_1999} p.115) For any prime number $p$ and any integers $k$ and $n$ such that $k \geq n \geq 1$, there exists an abelian group $G$ with order of $G$ equal to $p^n$ such that the set $D(p^k,p^k,G)$ is non-empty.
\end{prop}
\begin{proof}

Let $\mathbb{F}_{p^k}$ be the finite field with $p^k$ elements. Let the elements be represented by polynomials

$$\beta_0+\beta_1x+\cdots+\beta_{n-1}x^{n-1}+\cdots+\beta_{k-1}x^{k-1}$$
where $\beta_0,\ldots,\beta_{k-1} \in \mathbb{Z}_p$.

We may regard the finite field $\mathbb{F}_{p^n}$ as an additive subgroup of $\mathbb{F}_{p^k}$ by identifying its elements with polynomials of the form $\beta_0+\beta_1x+\cdots+\beta_{n-1}x^{n-1}$. (The multiplication of elements in $\mathbb{F}_{p^n}$ is in general different from the one in $\mathbb{F}_{p^k}$ but it is not a problem here, because we will only use the additive structure of $\mathbb{F}_{p^n}$).

Let $D^*$ be the multiplication table of $\mathbb{F}_{p^k}$ and let $\phi : \mathbb{F}_{p^k} \rightarrow \mathbb{F}_{p^n}$ be the map which maps the element $\beta_0+\beta_1x+\cdots+\beta_{k-1}x^{k-1}$ to the element $\beta_0+\beta_1x+\cdots+\beta_{n-1}x^{n-1}$.

We apply $\phi$ to each element of the table $D^*$ and we let $D$ denote the new table obtained in this way. Then $D$ is a difference matrix of $D(p^k,p^k,\mathbb{F}_{p^n})$.

Indeed, by construction, $D$ is a matrix of size $p^k \times p^k$ with entries in $\mathbb{F}_{p^n}$.

Let $\alpha_0,\ldots,\alpha_{p^k-1}$ be the elements of $\mathbb{F}_{p^k}$. Then, the difference of two columns of $D$ will have the form

\begin{center}
$\begin{pmatrix}
\phi(\beta\alpha_0) \\
\vdots \\
\phi(\beta\alpha_{p^k-1})
\end{pmatrix}
-
\begin{pmatrix}
\phi(\gamma\alpha_0) \\
\vdots \\
\phi(\gamma\alpha_{p^k-1})
\end{pmatrix}$
\end{center}
where $\beta,\gamma \in \mathbb{F}_{p^k},\beta \neq \gamma$.

Moreover, by definition of $\phi$ we have $\phi(\beta\alpha_i)-\phi(\gamma\alpha_i)=\phi(\beta\alpha_i-\gamma\alpha_i)$. The difference of two columns is equal to

\begin{center}
$\begin{pmatrix}
\phi((\beta-\gamma)\alpha_0) \\
\vdots \\
\phi((\beta-\gamma)\alpha_{p^k-1})
\end{pmatrix}$.
\end{center}

As each element of $\mathbb{F}_{p^k}$ appears once among the elements $(\beta-\gamma)\alpha_i, \ 0 \leq i < p^k$, every element of $\mathbb{F}_{p^k}$ appears $p^{k-n}$ times among the elements $\phi((\beta-\gamma)\alpha_i), \ 0 \leq i < p^k$.
\end{proof}

\begin{ex}
From the table of the finite field $\mathbb{F}_8 \simeq \mathbb{F}_2[X] \slash (X^3+X+1)$, we obtain the following matrix of $D(8,8,\mathbb{Z}^3_2)$:
$$\begin{pmatrix}
(0,0,0) & (0,0,0) & (0,0,0) & (0,0,0) & (0,0,0) & (0,0,0) & (0,0,0) & (0,0,0) \\
(0,0,0) & (0,0,1) & (0,1,0) & (0,1,1) & (1,0,0) & (1,0,1) & (1,1,0) & (1,1,1) \\
(0,0,0) & (0,1,0) & (1,0,0) & (1,1,0) & (0,1,1) & (0,0,1) & (1,1,1) & (1,0,1) \\
(0,0,0) & (0,1,1) & (1,1,0) & (1,0,1) & (1,1,1) & (1,0,0) & (0,0,1) & (0,1,0) \\
(0,0,0) & (1,0,0) & (0,1,1) & (1,1,1) & (1,1,0) & (0,1,0) & (1,0,1) & (0,0,1) \\
(0,0,0) & (1,0,1) & (0,0,1) & (1,0,0) & (0,1,0) & (1,1,1) & (0,1,1) & (1,1,0) \\
(0,0,0) & (1,1,0) & (1,1,1) & (0,0,1) & (1,0,1) & (0,1,1) & (0,1,0) & (1,0,0)\\
(0,0,0) & (1,1,1) & (1,0,1) & (0,1,0) & (0,0,1) & (1,1,0) & (1,0,0) & (0,1,1)\\
\end{pmatrix}$$
\end{ex}

\begin{ex}\label{ex4}
Hedayat, Sloane, and Stufken~\cite[p.117]{hedayat_orthogonal_1999} give an example of a matrix in $D(9,9,\mathbb{Z}_3)$ from the table of the finite field $\mathbb{F}_9 \simeq \mathbb{F}_3[X] \slash (X^2+1)$ with $9$ elements:

$$\begin{pmatrix}
0 & 0 & 0 & 0 & 0 & 0 & 0 & 0 & 0 \\
0 & 1 & 2 & 0 & 1 & 2 & 0 & 1 & 2 \\
0 & 2 & 1 & 0 & 2 & 1 & 0 & 2 & 1 \\
0 & 0 & 0 & 2 & 2 & 2 & 1 & 1 & 1 \\
0 & 1 & 2 & 2 & 0 & 1 & 1 & 2 & 0 \\
0 & 2 & 1 & 2 & 1 & 0 & 1 & 0 & 2 \\
0 & 0 & 0 & 1 & 1 & 1 & 2 & 2 & 2 \\
0 & 1 & 2 & 1 & 2 & 0 & 2 & 0 & 1 \\
0 & 2 & 1 & 1 & 0 & 2 & 2 & 1 & 0
\end{pmatrix}$$
\end{ex}

The existence of difference matrices has been extensively studied. Proposition~\ref{MatDif} gives a method for building explicitly a difference matrix with given parameters. However, not all difference matrices are obtained in this way.

Lampio and Östergård~\cite{lampio_classification_nodate,LampioClassificationdifferencematrices2011} propose a classification of difference matrices. It is based on an equivalence relation in the set of all difference matrices, defined by the following operations that generate a difference matrix with the same parameters (the numbers of rows, the numbers of columns, and the underlying group).

\begin{enumerate}
\item Permuting the order of rows.
\item Permuting the order of columns.
\item Adding a fixed element of the group $G$ to a row.
\item Adding a fixed element of the group $G$ to a column.
\item Applying an automorphism of the group $G$ to every element in the difference matrix.
\end{enumerate}

\begin{Def}[\cite{LampioClassificationdifferencematrices2011}]
We say that two difference matrices $A$ and $B$ are equivalent, denoted by $A \cong B$, if they have the same parameters and $B$ can be generated from $A$ by applying Operations 1-5 a finite number of times.
\end{Def}

The relation $\cong$ is an equivalence relation in the set of all difference matrices, and each equivalence class is a subset of the set of difference matrices with the same parameters.

\begin{Def}[\cite{LampioClassificationdifferencematrices2011}]
Let $G$ be an abelian group with some total order $\leq_G$ on the elements, where the identity element of $G$ is the minimal element. A difference matrix of $D(r,c,G)$ is an \textit{order-normalized difference matrix} if

\begin{enumerate}
\item the first row contains only the identity element,
\item the first column contains only the identity element,
\item the rows are in ascending lexicographic order from top to bottom (imposed by $\leq_G$ on row vectors), and
\item the columns are in ascending lexicographic order from left to right (imposed by $\leq_G$ on column vectors).
\end{enumerate}
\end{Def}

\begin{theorem}[\cite{LampioClassificationdifferencematrices2011}]
Every difference matrix of $D(r,c,G)$ is equivalent to an order-normalized difference matrix of $D(r,c,G)$.
\end{theorem}

The proof consists in using Operations 1,2,3 and 4 that define the equivalence relation in order to build an order-normalized difference matrix from a given difference matrix of $D(r,c,G)$.

\begin{rem}
This result implies that it suffices to study only order-normalized difference matrices to investigate the existence of a difference matrix with given parameters.
\end{rem}

\begin{rem}
The proof of Proposition~\ref{MatDif} gives a construction of difference matrices which already meet conditions 1 and 2 in the definition of order-normalized difference matrices. Then, by permuting rows and columns we can obtain the order-normalized difference matrices that are in the same equivalence class.
\end{rem}

Table 2 of~\cite{LampioClassificationdifferencematrices2011} gives the number of equivalence classes of difference matrices according to the parameters.

\begin{ex}\label{ex3} In $D(9,9,\mathbb{Z}_3)$, there are two equivalence classes of difference matrices. A representative of each equivalence class is given in~\cite{LampioClassificationdifferencematrices2011}:
$$\begin{pmatrix}
0 & 0 & 0 & 0 & 0 & 0 & 0 & 0 & 0 \\
0 & 0 & 0 & 1 & 1 & 1 & 2 & 2 & 2 \\
0 & 0 & 0 & 2 & 2 & 2 & 1 & 1 & 1 \\
0 & 1 & 2 & 0 & 1 & 2 & 0 & 1 & 2 \\
0 & 1 & 2 & 1 & 2 & 0 & 2 & 0 & 1 \\
0 & 1 & 2 & 2 & 0 & 1 & 1 & 2 & 0 \\
0 & 2 & 1 & 0 & 2 & 1 & 0 & 2 & 1 \\
0 & 2 & 1 & 1 & 0 & 2 & 2 & 1 & 0 \\
0 & 2 & 1 & 2 & 1 & 0 & 1 & 0 & 2
\end{pmatrix}
\qquad \qquad
\begin{pmatrix}
0 & 0 & 0 & 0 & 0 & 0 & 0 & 0 & 0 \\
0 & 0 & 0 & 1 & 1 & 1 & 2 & 2 & 2 \\
0 & 0 & 0 & 2 & 2 & 2 & 1 & 1 & 1 \\
0 & 1 & 2 & 0 & 1 & 2 & 0 & 1 & 2 \\
0 & 1 & 2 & 1 & 2 & 0 & 2 & 0 & 1 \\
0 & 1 & 2 & 2 & 0 & 1 & 1 & 2 & 0 \\
0 & 2 & 1 & 0 & 2 & 1 & 1 & 0 & 2 \\
0 & 2 & 1 & 1 & 0 & 2 & 0 & 2 & 1 \\
0 & 2 & 1 & 2 & 1 & 0 & 2 & 1 & 0
\end{pmatrix}$$
\end{ex}

\begin{rem}
By permuting the rows, the matrix obtained in Example~\ref{ex4} is equivalent to the order-normalized difference matrix on the left in Example~\ref{ex3}. Therefore, the matrix of the second equivalence class is necessarily obtained otherwise.
\end{rem}

We are now ready to define a generalization of the Rudin--Shapiro sequence via Proposition~\ref{MatDif}. It is an extension of the generalization in Definition~\ref{def3} for powers of prime numbers. 

\begin{Def}\label{def7}
Let $p$ be a prime number, let $k \geq 1$ and let $M=(m_{ij})_{\substack{0 \leq i < p^k \\ 0 \leq j < p^k}}$ be a difference matrix of $D(p^k,p^k,\mathbb{Z}^k_p)$. Let

\begin{align*}
g : \mathbb{Z} \times \mathbb{Z} &\longrightarrow \mathbb{Z}^k_p\\
(j,n) &\longmapsto m_{n \ \text{mod} \ p^k, \ j \ \text{mod} \ p^k}
\end{align*}

We let $g_1,\ldots,g_k$ denote the functions with values in $\mathbb{Z}_p$ such that
\\ $g(j,n)=(g_1(j,n),\ldots,g_k(j,n))$.

We say that the sequence defined by $(a(n))_{n \geq 0}=(a_1(n),\ldots,a_k(n))_{n \geq 0}$ and
$$a(p^k n+j)=a(n)+g(j,n), \quad 0 \leq j \leq p^k-1, \quad n \geq 0,  \quad (j,n)\neq(0,0)$$
is the \emph{Rudin--Shapiro sequence associated to the matrix $M$}.
\end{Def}

\begin{rem}
We can fix arbitrarily the value of $a(0)$ and the other terms are defined recursively.
\end{rem}

\begin{rem}
When the size of the alphabet is $p$, with $p$ a prime, Definition~\ref{def3} and Definition~\ref{def7} coincide, except possibly for the $p$ first values of the sequence.
\end{rem}

\begin{rem}
By definition of $g$, for all integers $u$ and $i$ with $0 \leq u < u+i \leq p^k-1$ the set $\left\{(g(u+i,n)-g(u,n)): 0 \leq n \leq p^k-1\right\}$ is equal to the set of the elements of $\mathbb{Z}^k_p$.
\end{rem}

\section{Main results}\label{Section4}

We have already seen results about the correlation of order 2 in the case where the size of the alphabet is a prime number or a squarefree product of prime numbers (Theorem~\ref{theo1} and Theorem~\ref{theo2}). In this part, we give a similar result for an alphabet of any size. First, we give a result for the alphabets whose size is a power of a prime number. The proof follows the lines of Theorem~\ref{theo1}, we give the full details in Section~\ref{Section5} for a better understanding and in order that the paper is self-contained.

\begin{theorem}\label{theo4}
Let $p$ be a prime number and $k \geq 1$. Let $M$ be a difference matrix in $D(p^k,p^k,\mathbb{Z}^k_p)$ and let $(a(n))_{n \geq 0}$ be the Rudin--Shapiro sequence associated to $M$.  Moreover, let $0 \leq r_1 < r_2$. Then, as $N \rightarrow \infty$, we have

$$\displaystyle\sum_{n<N}\delta(n+r_1,n+r_2)=N\left(1-\cfrac{1}{p^k}\right)+O_{p,k}\left((r_2-r_1)\text{log}\frac{N}{r_2-r_1}+r_2\right).$$

\end{theorem}

\begin{ex}\label{ex6}
Let $(\tilde{a}(n))_{n \geq 0}$ be the sequence obtained from the generalized Rudin--Shapiro sequence $(a(n))_{n \geq 0}$ associated to the matrix~\eqref{matrix} over $D(4,4,\mathbb{Z}_2 \times \mathbb{Z}_2)$ by recoding $(0,0)$ to $0$, $(0,1)$ to $1$, $(1,0)$ to $2$ and $(1,1)$ to $3$. So, $(\tilde{a}(n))_{n \geq 0}$ is a sequence over the alphabet $\left\{0,1,2,3\right\}$, whose first terms are given below.
$$(\tilde{a}(n))_{n \geq 0}=0,0,0,0,0,1,2,3,0,2,3,1,0,3,1,2,0,0,0,0,1,0,3,2,2,0,1,3,\ldots$$  
Moreover, let $0 \leq r_1 < r_2$. Then, as $N \rightarrow \infty$, we have

$$\displaystyle\sum_{n<N}\delta(n+r_1,n+r_2)=\frac{3}{4}N+O\left((r_2-r_1)\text{log}\frac{N}{r_2-r_1}+r_2\right).$$
\end{ex}

\begin{rem}
It is possible to use a similar recoding for any choice of $p^k$.
\end{rem}

We have also the following corollary.

\begin{cor}
In the setting of Theorem~\ref{theo4}, if $r_2=o(N)$ then

$$\displaystyle\sum_{n<N}\delta(n+r_1,n+r_2)\sim N\left(1-\cfrac{1}{p^k}\right).$$

Consequently, in Example~\ref{ex6}, for $r_2=o(N)$, we have the same result as Grant et al. for an alphabet of size 4,

$$\displaystyle\sum_{n<N}\delta(n+r_1,n+r_2)\sim \frac{3}{4}N.$$
\end{cor}

Now, we present the general case for any alphabet. 

\begin{theorem}\label{theo5}
Let $d \geq 2$, and let $p_1,\ldots,p_d$ be pairwise distinct primes and $k_1,\ldots,k_d$ positive integers. We consider the alphabet $\left\{0,\ldots,k-1\right\}$, where $k=p_1^{k_1}\cdots p_d^{k_d}$.

For every $1 \leq i \leq d$, we consider a difference matrix $M_i$ of $D(p_i^{k_i},p_i^{k_i},\mathbb{Z}^{k_i}_{p_i})$, to which we associate a function $g^{i}(j,n)=(g^{i}_1(j,n),\ldots,g^{i}_{k_i}(j,n))$ and a sequence $a^{i}(n)=(a^{i}_1(n),\ldots,a^{i}_{k_i}(n))$ as previously defined.
We define the sequence $(\hat{a}(n))_{n \geq 0}$ by 

$$\hat{a}(n)=(a^{1}(n) \ \text{mod} \ p_1,\ldots, a^{d}(n) \ \text{mod} \ p_d).$$

Moreover, let $0 \leq r_1 < r_2$. Then, as $N \rightarrow \infty$, we have

$$\displaystyle\sum_{n<N}\delta(n+r_1,n+r_2)=N\left(1-\cfrac{1}{k}\right)+O_k\left(\left((r_2-r_1)\text{log}\frac{N^{\frac{1}{d}}}{r_2-r_1}+r_2\right)N^{\frac{d-1}{d}}\right).$$
\end{theorem}

In the same way as before, we obtain the following corollary.

\begin{cor}
In the setting of Theorem~\ref{theo4}, if $r_2=o(N^{\frac{1}{d}})$ then
$$\displaystyle\sum_{n<N}\delta(n+r_1,n+r_2)\sim N\left(1-\cfrac{1}{k}\right).$$
\end{cor}

\begin{rem}
By comparing the error terms of Theorems~\ref{theo2} and~\ref{theo5} when the size of the alphabet is squarefree, we observe that when $r_2-r_1=O(1)$, the optimal choice of $\gamma$ in Theorem~\ref{theo2} is achieved when $1-\frac{\gamma}{d}=\gamma$, i.e., $\gamma=\frac{d}{d+1}$. This gives an error term bound by $N^{\frac{d}{d+1}}$. In Theorem~\ref{theo5}, the corresponding error term is bound by $r_2N^{\frac{d-1}{d}}$, therefore, in order to obtain an improvement we need $r_2N^{\frac{d-1}{d}} \ll N^{\frac{d}{d+1}}$, i.e., $r_2=o(N^{\frac{1}{d(d+1)}})$. Thus, if $r_2-r_1=O(1)$ and $r_2=o(N^{\frac{1}{d(d+1)}})$, our result is an improvement for the alphabets where the size is squarefree and with at least two prime numbers.
\end{rem}

\section{Proofs}\label{Section5}

\subsection{Proof of Theorem~\ref{theo4}}

For the proof of Theorem~\ref{theo4}, we need the following lemma.

\begin{lem}\label{lem1}
Let $G$ be a difference matrix of $D(p^k,p^k,\mathbb{Z}^k_p)$. We let $G_1,\ldots,G_k$ denote the matrices obtained from $G$ by taking respectively the first,$\ldots$, the $k$-th coordinate. Let $0 \leq h_1,\ldots,h_k<p$ with $(h_1,\ldots,h_k)\neq (0,\ldots, 0)$. Then the matrix  $H=h_1G_1+\cdots+h_kG_k$ is a difference matrix of $D(p^k,p^k,\mathbb{Z}_p)$.
\end{lem}

\begin{proof}
We let $(g_1(j,n),\ldots,g_k(j,n))$ denote the element of $G$ at the $j$-th column and the $n$-th row. The difference between two distinct columns $i$ and $j$ of $H$ can be written as

\begin{small}
\begin{align*}
C_{i,j}=\begin{pmatrix}
h_1(g_1(j,0)-g_1(i,0))+\cdots+h_k(g_k(j,0)-g_k(i,0)) \\
\vdots \\
h_1(g_1(j,p^k-1)-g_1(i,p^k-1))+\cdots+h_k(g_k(j,p^k-1)-g_k(i,p^k-1))
\end{pmatrix}.
\end{align*}
\end{small}

As $G$ is a difference matrix, we have 
$$\left\{(g_1(j,n)-g_1(i,n),\ldots,g_k(j,n)-g_k(i,n)), \ 0 \leq n<p^k\right\}=\mathbb{Z}^k_p.$$

Therefore, the elements that appear in $C_{i,j}$ are all the elements of the form $h_1c_1+\cdots+h_kc_k$, for $(c_1,\ldots,c_k) \in \mathbb{Z}^k_p$. Thus, in $C_{i,j}$, for all $d \in \mathbb{Z}_p$, each element appears 

$\#\left\{(c_1,\ldots,c_k)\in \mathbb{Z}^k_p:h_1c_1+\cdots+h_kc_k=d\right\}=p^{k-1}$ times. Consequently, $H$ is a difference matrix of $D(p^k,p^k,\mathbb{Z}_p)$.
\end{proof}

Now, we have all the tools to prove Theorem~\ref{theo4}.

\begin{proof}
Let $0 \leq r_1 < r_2$. We have

\begin{align*}
&\displaystyle\sum_{n<N}\delta(n+r_1,n+r_2)\\
&=N-\displaystyle\sum_{n<N}\cfrac{1}{p^k}\displaystyle\prod_{i=1}^{k}\displaystyle\sum_{0\leq h_i<p}\text{e}\left(\cfrac{h_i}{p} (a_i(n+r_2)-a_i(n+r_1))\right)\\
&=N-\displaystyle\sum_{n<N}\cfrac{1}{p^k}\displaystyle\sum_{0 \leq h_1,\ldots,h_k<p}\text{e}\left(\cfrac{1}{p}\displaystyle\sum_{i=1}^{k} h_i(a_i(n+r_2)-a_i(n+r_1))\right)\\
&=N\left(1-\cfrac{1}{p^k}\right)-\cfrac{1}{p^k}\displaystyle\sum_{\substack{0 \leq h_1,\ldots,h_k<p \\ (h_1,\ldots,h_k)\neq (0,\ldots, 0)}}S_N(h_1,\ldots,h_k),
\end{align*}
with 

$$S_N(h_1,\ldots,h_k)=\displaystyle\sum_{n<N}\text{e}\left(\cfrac{1}{p}\displaystyle\sum_{i=1}^{k} h_i(a_i(n+r_2)-a_i(n+r_1))\right).$$

Put $r=r_2-r_1$.

It suffices to show that for all $0 \leq h_1,\ldots,h_k<p$ with $(h_1,\ldots,h_k)\neq (0,\ldots, 0)$ we have 

$$S_N(h_1,\ldots,h_k)=O_{p,k}\left(r\text{log}\cfrac{N}{r}+r\right).$$

Let $b(n)=h_1a_1(n)+\cdots+h_ka_k(n)$ and $g^*(j,n)=h_1g_1(j,n)+\cdots+h_kg_k(j,n)$ so that $b(p^kn+j)=b(n)+g^*(j,n)$.

By Lemma~\ref{lem1}, for all integers $u$ and $i$ such that $0 \leq u < u+i \leq p^k-1$, the set $\left\{(g^*(u+i,n)-g^*(u,n)) \ : 0 \leq n \leq p^k-1\right\}$ contains $p^{k-1}$ times each element of $\mathbb{Z}_p$.

We define

$$\gamma_N(r,f)=\displaystyle\sum_{n<N}\text{e}\left(\cfrac{b(n+r)-b(n)}{p}\right)\text{e}\left(\cfrac{f(n)}{p}\right),$$
where $f : \mathbb{N} \rightarrow \mathbb{Z}$ is an arbitrary periodic function with period $p^k$.

Let us begin by showing that $\gamma_N(1,f)=O(\text{log} N)$ for $N>p^k$. In order to show this, we decompose $n$ modulo $p^k$. For this purpose, we replace $N$ by $p^kN+j$, with $0 \leq j \leq p^k-1$. Then, we have
\begin{align}
\gamma_{p^kN+j}(1,f) &=\displaystyle\sum_{n<p^kN+j}\text{e}\left(\cfrac{1}{p}(b(n+1)-b(n))\right)\text{e}\left(\cfrac{f(n)}{p}\right)\nonumber\\
&=\displaystyle\sum_{u=0}^{p^k-1}\displaystyle\sum_{p^kn+u<p^kN+j}\text{e}\left(\cfrac{1}{p}(b(p^kn+u+1)-b(p^kn+u))\right)\text{e}\left(\cfrac{f(u)}{p}\right)\nonumber\\
&=\displaystyle\sum_{u=0}^{j-1}\text{e}\left(\cfrac{1}{p}(b(p^kN+u+1)-b(p^kN+u))\right)\text{e}\left(\cfrac{f(u)}{p}\right)\label{triv}\\
&+\displaystyle\sum_{u=0}^{p^k-2}\text{e}\left(\cfrac{f(u)}{p}\right)\displaystyle\sum_{0\leq n<N}\text{e}\left(\cfrac{1}{p}(b(p^kn+u+1)-b(p^kn+u))\right)\label{summaj1}\\
&+\text{e}\left(\cfrac{f(p^k-1)}{p}\right)\displaystyle\sum_{0\leq n<N}\text{e}\left(\cfrac{1}{p}(b(p^kn+p^k)-b(p^kn+p^k-1)\right)\label{summaj2}.
\end{align}

The term \eqref{triv} is trivially bounded by $j \leq p^k-1$.

For \eqref{summaj1} we have for $0 \leq u \leq p^k-2$,

$\displaystyle\sum_{0\leq n<N}\text{e}\left(\cfrac{1}{p}(b(p^kn+u+1)-b(p^kn+u))\right)$
\begin{align*}
&=\displaystyle\sum_{0\leq n<N}\text{e}\left(\cfrac{1}{p}(b(n)+g^*(u+1,n)-b(n)-g^*(u,n))\right)\\
&=\displaystyle\sum_{0\leq n<N}\text{e}\left(\cfrac{1}{p}(g^*(u+1,n)-g^*(u,n))\right).
\end{align*}

For $0 \leq n \leq p^k-1$ and fixed $u$, the differences $g^*(u+1,n)-g^*(u,n)$ take $p^{k-1}$ times every value of $\mathbb{Z}_p$. Therefore, this sum is bounded by $\cfrac{p^k}{2}$. Consequently, the sum~\eqref{summaj1} is bounded by $\cfrac{(p^k-1)p^k}{2}$.

Finally, for \eqref{summaj2} we have

$\displaystyle\sum_{0\leq n<N}\text{e}\left(\cfrac{1}{p}(b(p^kn+p^k)-b(p^kn+p^k-1)\right)$
\begin{align*}
&=\displaystyle\sum_{0\leq n<N}\text{e}\left(\cfrac{1}{p}(b(n+1)+g^*(0,n+1)-b(n)-g^*(p^k-1,n)\right)\\
&=\displaystyle\sum_{0\leq n<N}\text{e}\left(\cfrac{1}{p}(b(n+1)-b(n))\right)\text{e}\left(\cfrac{\tilde{f}(n)}{p}\right),
\end{align*}
where $\tilde{f}(n)=g^*(0,n+1)-g^*(p^k-1,n)$ is periodic with period $p^k$.

We deduce that $|\gamma_{p^kN+j}(1,f)| \leq |\gamma_{N}(1,\tilde{f})| + \cfrac{(p^k-1)(p^k+2)}{2}$.

Moreover, since $|\gamma_{n}(1,f)| \leq p^k-1$ for $1 \leq n \leq p^k-1$ and all periodic functions $f$ with period $p^k$, it follows by induction that for all periodic functions $f$ with period $p^k$ and for all $N>p^k$,

\begin{equation}
|\gamma_{N}(1,f)| \leq \cfrac{(p^k-1)(p^k+2)}{2k\text{log}p} \ \text{log}N+p^k-1.\label{maj1}
\end{equation}

Indeed, suppose that for $N>p^k$ we have~\eqref{maj1} for all periodic functions $f$ with period $p^k$. Then, let $f$ be a periodic function with period $p^k$ and $0 \leq j \leq p^k-1$. We have 
\begin{align*}
|\gamma_{p^kN+j}(1,f)|
&\leq|\gamma_{N}(1,\tilde{f})|+\cfrac{(p^k-1)(p^k+2)}{2}\\
&\leq \cfrac{(p^k-1)(p^k+2)}{2k\text{log}p} \ \text{log}N+p^k-1+\cfrac{(p^k-1)(p^k+2)}{2}\\
&\leq \cfrac{(p^k-1)(p^k+2)}{2k\text{log}p} \ (\text{log}N+k\text{log}p)+p^k-1\\
&\leq \cfrac{(p^k-1)(p^k+2)}{2k\text{log}p} \ \text{log}(p^kN+j)+p^k-1.
\end{align*}

We note that the sum $\gamma_N(0,f)=\displaystyle\sum_{n<N}\text{e}\left(\cfrac{f(n)}{p}\right)$ satisfies 

\begin{equation}
|\gamma_{N}(0,f)| \leq \cfrac{p^k}{2} \ \text{if} \ f(\left\{0,\ldots,p^k-1\right\}) \ \text{contains} \ p^{k-1} \ \text{times each element of} \ \mathbb{Z}_p.\label{maj2}
\end{equation}

Now, let us consider the general case with $r=p^kM+i > 0$ where $M \geq 0$ and $0 \leq i \leq p^k-1$ but $(M,i)\neq (0,0)$. We have
\begin{align}
\gamma_{p^kN+j}&(p^kM+i,f)\nonumber\\
&=\displaystyle\sum_{n<p^kN+j}\text{e}\left(\cfrac{1}{p}(b(n+p^kM+i)-b(n))\right)\text{e}\left(\cfrac{f(n)}{p}\right)\nonumber\\
&=\displaystyle\sum_{n<p^kN}\text{e}\left(\cfrac{1}{p}(b(n+p^kM+i)-b(n))\right)\text{e}\left(\cfrac{f(n)}{p}\right)+O_{p,k}(1)\nonumber\\
&=\displaystyle\sum_{u=0}^{p^k-1}\displaystyle\sum_{0\leq n<N}\text{e}\left(\cfrac{1}{p}(b(p^kn+u+p^kM+i)-b(p^kn+u))\right)\text{e}\left(\cfrac{f(u)}{p}\right)\nonumber\\
&\quad +O_{p,k}(1)\nonumber\\
&=\displaystyle\sum_{u=0}^{p^k-1}\text{e}\left(\cfrac{f(u)}{p}\right)\displaystyle\sum_{0\leq n<N}\text{e}\left(\cfrac{1}{p}(b(p^kn+u+p^kM+i)-b(p^kn+u))\right)\label{form1}\\
&\quad +O_{p,k}(1)\nonumber,
\end{align}
where the implied constant comes from the terms $n=N$ and is bounded by $p^k-1$. The last part consists in estimating the sum given in \eqref{form1}. First, we suppose that $i \neq 0$. Then

$\displaystyle\sum_{u=0}^{p^k-1}\text{e}\left(\cfrac{f(u)}{p}\right)\displaystyle\sum_{0\leq n<N}\text{e}\left(\cfrac{1}{p}(b(p^kn+u+p^kM+i)-b(p^kn+u))\right)$
\begin{align*}
&=\displaystyle\sum_{u=0}^{p^k-1-i}\text{e}\left(\cfrac{f(u)}{p}\right)\displaystyle\sum_{0\leq n<N}\text{e}\left(\cfrac{1}{p}(b(n+M)+g^*(u+i,n+M)-b(n)-g^*(u,n))\right)\\ 
&+\displaystyle\sum_{u=p^k-i}^{p^k-1}\text{e}\left(\cfrac{f(u)}{p}\right)\\
&\times\displaystyle\sum_{0\leq n<N}\text{e}\left(\cfrac{1}{p}(b(n+M+1)+g^*(u+i-p^k,n+M+1)-b(n)-g^*(u,n))\right)\\ 
&=\displaystyle\sum_{u=0}^{p^k-1-i}\text{e}\left(\cfrac{f(u)}{p}\right)\displaystyle\sum_{0\leq n<N}\text{e}\left(\cfrac{1}{p}(b(n+M)-b(n))\right)\text{e}\left(\cfrac{f_1(n)}{p}\right)\\
&+\displaystyle\sum_{u=p^k-i}^{p^k-1}\text{e}\left(\cfrac{f(u)}{p}\right)\displaystyle\sum_{0\leq n<N}\text{e}\left(\cfrac{1}{p}(b(n+M+1)-b(n))\right)\text{e}\left(\cfrac{f_2(n)}{p}\right),
\end{align*}
with $f_1(n)=g^*(u+i,n+M)-g^*(u,n)$ for $0 \leq u \leq p^k-1-i$,\\
and $f_2(n)=g^*(u+i-p^k,n+M+1)-g^*(u,n)$ for $p^k-i \leq u \leq p^k-1$.

For the sake of simplicity, here and later on, we do not write down the dependency on $u$ of these functions. Thus

\begin{align*}
\left|\displaystyle\sum_{u=0}^{p^k-1}\text{e}\left(\cfrac{f(u)}{p}\right)\displaystyle\sum_{0\leq n<N}\text{e}\left(\cfrac{1}{p}(b(p^kn+u+p^kM+i)-b(p^kn+u))\right)\right|\\
\leq\left|\displaystyle\sum_{u=0}^{p^k-1-i}\text{e}\left(\cfrac{f(u)}{p}\right)\gamma_N(M,f_1)\right|+\left|\displaystyle\sum_{u=p^k-i}^{p^k-1}\text{e}\left(\cfrac{f(u)}{p}\right)\gamma_N(M+1,f_2)\right|.
\end{align*}
Let $\tilde{f_1}$ and $\tilde{f_2}$ be two functions such that $|\gamma_N(M,\tilde{f_1})|=\displaystyle\max_{0\leq u \leq p^k-1-i}|\gamma_N(M,f_1)|$ and $|\gamma_N(M,\tilde{f_2})|=\displaystyle\max_{p^k-i\leq u \leq p^k-1}|\gamma_N(M,f_2)|$. We deduce the following estimate:

\begin{equation}
|\gamma_{p^kN+j}(p^kM+i,f)| \leq (p^k-i)|\gamma_N(M,\tilde{f_1})|+i|\gamma_N(M+1,\tilde{f_2})|+p^k-1.\label{maj3}
\end{equation}

Let us substitute $M=0$ in \eqref{maj3}.

Since $i \neq 0$, the image of the set $\left\{0,\ldots,p^{k-1}\right\}$ by the function $f_1(n)=g^*(u+i,n)-g^*(u,n)$ is the multiset $\{\underbrace{0,\ldots,0}_{p^{k-1}},\ldots,\underbrace{p-1,\ldots,p-1}_{p^{k-1}}\}$. 

Using \eqref{maj1} and \eqref{maj2} we therefore get

$$(p^k-i)|\gamma_N(0,\tilde{f_1})| \leq (p^k-i)\times \cfrac{p^k}{2}.$$
and

$$i|\gamma_N(1,\tilde{f_2})| \leq i\left(\cfrac{(p^k-1)(p^k+2)}{2k\text{log}p} \ \text{log}N+p^k-1\right).$$

Therefore,
\begin{align*}
|\gamma_{p^kN+j}(i,f)| &\leq (p^k-i)\cfrac{p^k}{2}+i\left(\cfrac{(p^k-1)(p^k+2)}{2k\text{log}p} \  \text{log}N+p^k-1\right)+p^k-1\\
&\leq i\left(\cfrac{(p^k-1)(p^k+2)}{2k\text{log}p} \ \text{log}N\right)+(p^k-i)\cfrac{p^k}{2}+i(p^k-1)+p^k-1\\
&\leq \cfrac{(p^k-1)^2(p^k+2)}{2k\text{log}p} \ \text{log}N+\cfrac{p^k}{2}(p^k-i+2i+2)-i-1\\
&\leq \cfrac{(p^k-1)^2(p^k+2)}{2k\text{log}p} \ \text{log}N+\cfrac{p^k}{2}(2p^k+1)-p^k.
\end{align*}

Thus, for all $1\leq i \leq p^k-1$ and all periodic functions $f$ with period $p^k$, we have, for $N>p^k$,

\begin{equation}
|\gamma_{N}(i,f)| \leq \cfrac{(p^k-1)^2(p^k+2)}{2k\text{log}p} \ \text{log}\cfrac{N}{p^k}+\cfrac{p^k}{2}(2p^k+1)-p^k.\label{maj4}
\end{equation}
We now establish a bound for $i=0$. For $0 \leq u \leq p^k-1$ we have
$$b(p^kn+u+p^kM)-b(p^kn+u)=b(n+M)-b(n)+g^*(u,n+M)-g^*(u,n)$$
and therefore, for $M\neq 0$, by~\eqref{form1}
\begin{equation}
|\gamma_{p^kN+j}(p^kM,f)| \leq \displaystyle\sum_{u=0}^{p^k-1}|\gamma_N(M,f_3)|+p^k-1 \label{maj5}
\end{equation}
with $f_3(n)=g^*(u,n+M)-g^*(u,n)$. Using \eqref{maj1} and substituting $M=1$ in \eqref{maj5}, we deduce, for $N>p^k$,

$$|\gamma_{N}(p^k,f)| \leq p^k\left(\cfrac{(p^k-1)(p^k+2)}{2k\text{log}p} \ \text{log}\cfrac{N}{p^k}+p^k\right).$$

Hence using \eqref{maj4}, for all $N>p^k$ and for $1 \leq i \leq p^k$,

\begin{equation}
|\gamma_{N}(i,f)| \leq p^k\left(\cfrac{(p^k-1)(p^k+2)}{2k\text{log}p} \ \text{log}\cfrac{N}{p^k}+p^k\right).\label{maj6}
\end{equation}

Using \eqref{maj3}, we have for $1 \leq i \leq p^k-1$, $0 \leq m \leq p^{k(s-1)}(p^k-1)-1$, $M=p^{k(s-1)}+m$ with $s \geq 1$, and for all $N>p^{k(s+1)}$,

\begin{align*}
|\gamma_{p^kN+j}&(p^k(p^{k(s-1)}+m)+i,f)|\\
&\leq (p^k-i)|\gamma_{N}(p^{k(s-1)}+m,\tilde{f_1})|+i|\gamma_{N}(p^{k(s-1)}+m+1,\tilde{f_2})|+p^k-1\\
&\leq p^k\max(|\gamma_{N}(p^{k(s-1)}+m,\tilde{f_1})|,|\gamma_{N}(p^{k(s-1)}+m+1,\tilde{f_2})|)+p^k-1.
\end{align*}

Let $N=p^kN_1+j_1$. Depending on whether $p^k$ is a factor or not of $p^{k(s-1)}+m$ (resp. $p^{k(s-1)}+m+1$), we can use \eqref{maj3} or \eqref{maj5} to bound $|\gamma_{p^kN_1+j_1}(p^{k(s-1)}+m,\tilde{f_1})|$ (resp. $|\gamma_{p^kN_1+j_1}(p^{k(s-1)}+m+1,\tilde{f_2})|$). By iterating $s$ times, and using \eqref{maj6} for the last bound, we obtain for $r=p^{ks}+1,\ldots,p^{ks}+p^k-1,p^{ks}+p^k+1,\ldots,p^{k(s+1)}-p^k-1,p^{k(s+1)}-p^k+1,\dots,p^{k(s+1)}-1$ with $s \geq 1$, and for all $N>p^{k(s+1)}$,

\begin{align}
|\gamma_{N}(r,f)|\leq p^{ks}\left(p^k\cfrac{(p^k-1)(p^k+2)}{2k\text{log}p} \ \text{log}\cfrac{N}{p^{k(s+1)}}+p^k+1\right)+\displaystyle\sum_{j=0}^{s-1}(p^k-1)p^{kj}.\label{maj7}
\end{align}

For $r=p^{ks}+p^k,p^{ks}+2p^k,\ldots,p^{k(s+1)}$ we use \eqref{maj5}. Let $\tilde{f_3}$ be a function such that $|\gamma_N(M,\tilde{f_3})|=\displaystyle\max_{0\leq u \leq p^k-1}|\gamma_N(M,f_3)|$. Then, we have for all $N>p^{k(s+1)}$.

\begin{align*}
|\gamma_{p^kN +j}(r,f)| &\leq \displaystyle\sum_{u=0}^{p^k-1}|\gamma_{N}(\frac{r}{p^k},f_3)|+p^k-1\\
&\leq p^k|\gamma_{N}(\frac{r}{p^k},\tilde{f_3})|+p^k-1.
\end{align*}

We can then again iterate \eqref{maj3} or \eqref{maj5}, and \eqref{maj6} for the last bound. With \eqref{maj6} and \eqref{maj7}, we deduce for $r=p^{ks}+1,\ldots,p^{k(s+1)}$ with $s \geq 0$ and for all $N>p^{k(s+1)}$,

\begin{align*}
|\gamma_{N}(r,f)|&\leq p^{ks}\left(p^k\cfrac{(p^k-1)(p^k+2)}{2k\text{log}p} \ \text{log}\cfrac{N}{p^{k(s+1)}}+p^k+1\right)+\displaystyle\sum_{j=0}^{s-1}(p^k-1)p^{kj}\\
&\leq p^{ks}\left(p^k\cfrac{(p^k-1)(p^k+2)}{2k\text{log}p}\right)\text{log}\cfrac{N}{p^{k(s+1)}}+p^{ks}(p^k+2)-1.
\end{align*}

Finally, for all $N>rp^k$, we have 

$$|\gamma_{N}(r,f)| \leq r\left(p^k\cfrac{(p^k-1)(p^k+2)}{2k\text{log}p}\right)\text{log}\cfrac{N}{r}+r(p^k+2).$$

This completes the proof of Theorem~\ref{theo4}.
\end{proof}

\subsection{Proof of Theorem~\ref{theo5}}

Let $n \in \mathbb{N}$. We let $[\alpha_s,\alpha_{s-1},\ldots,\alpha_1,\alpha_0]_k$ denote the standard base-$k$ representation of $n$, where $\alpha_s \neq 0$ is the most significant digit, so that $n=\alpha_s k^s+\alpha_{s-1} k^{s-1}+\cdots+\alpha_1 k+\alpha_0$. We take the convention that $\alpha_{s+1}=0$. For the proof of Theorem~\ref{theo5} we will need the following elementary lemma:

\begin{lem}\label{lem2}
Let $k \geq 2$ and let $(a(n))_{n \geq 0}$ be a sequence associated to a generalized Rudin--Shapiro sequence, in the sense of Definition~\ref{def3}, which satisfies the relation 
$$a(nk+j)=a(n)+g(j,n), \quad 0\leq j \leq k-1, \quad n\geq 0, \quad (j,n)\neq(0,0).$$

Then, for $n=[\alpha_s,\alpha_{s-1},\ldots,\alpha_1,\alpha_0]_k$ we have 

$$a(n)=a(\alpha_s)+\displaystyle\sum_{i=0}^{s-1}g(\alpha_i,\alpha_{i+1})=a(0)+\displaystyle\sum_{i=0}^{s}g(\alpha_i,\alpha_{i+1}).$$
\end{lem}

\begin{proof}
By definition, the function $g$ is periodic in the second variable with period $k$. By induction on $s$, we have 
\begin{align*}
a(n)&=a(\alpha_s k^s+\alpha_{s-1}k^{s-1}+\cdots+\alpha_1 k+\alpha_0)\\
&=a(\alpha_s k^{s-1}+\alpha_{s-1} k^{s-2}+\cdots+\alpha_2 k+\alpha_1)+g(\alpha_0,\alpha_1)\\
&=\ldots =a(\alpha_s)+\displaystyle\sum_{i=0}^{s-1}g(\alpha_i,\alpha_{i+1}).\qedhere
\end{align*}

Now, since we have $a(\alpha_s)=a(0)+g(\alpha_s,0)=a(0)+g(\alpha_s,\alpha_{s+1})$, we deduce

$$a(n)=a(0)+\displaystyle\sum_{i=0}^{s}g(\alpha_i,\alpha_{i+1}).$$
\end{proof}

We end this section by the proof of Theorem~\ref{theo5}.

\begin{proof}
Let us begin with some notation.

We set $r=r_2-r_1$. Let $N$ be an integer and let $\textbf{b}=(b_1,\ldots,b_d)$, define

$$P_{\textbf{b}}=\left\{n \in \mathbb{N} : \forall i \in \{1,\ldots,d\}, n\equiv b_i \ (\text{mod} \ {p_i}^{s_i}) \right\},$$
where $s_i$ is the unique integer with ${p_i}^{s_i}\leq N^{\frac{1}{d}}<{p_i}^{s_i+1}$. As a first estimate, we have

$$\#\left\{n \in \mathbb{N} : n \in P_{\textbf{b}},\ n<N \right\}=\cfrac{N}{\prod_{i=1}^d{p_i}^{s_i}}+O(1).$$

We consider the sets

$$\mathcal{B}=\left\{(b_1,\ldots,b_d):0\leq b_i < {p_i}^{s_i}\right\},$$
$$\mathcal{B}_0=\left\{(b_1,\ldots,b_d):0\leq b_i < {p_i}^{s_i}-r\right\}.$$

Fix $1\leq i \leq d$ and $1 \leq j \leq k_i$. Now, consider $n$ such that $n=n_i {p_i}^{s_i}+b_i$ where $(b_1,\ldots,b_d)\in\mathcal{B}_0$. Write

$$b_i+r=\beta'_{s_i-1,i}{p_i}^{s_i-1}+\beta'_{s_i-2,i}{p_i}^{s_i-2}+\cdots+\beta'_{0,i},$$
$$b_i=\beta_{s_i-1,i}{p_i}^{s_i-1}+\beta_{s_i-2,i}{p_i}^{s_i-2}+\cdots+\beta_{0,i},$$
where $\beta_{\nu,i},\beta'_{\nu,i} \in \left\{0,1,\ldots,p_i-1\right\}$ for $0\leq \nu < s_i$. Moreover, consider

$$v_i=\text{max}(\kappa : \beta'_{\kappa,i} \neq 0, \ 0 \leq \kappa \leq s_i-1),$$
$$w_i=\text{max}(\kappa : \beta_{\kappa,i} \neq 0, \ 0 \leq \kappa \leq s_i-1),$$ 
which correspond to the uppermost non-zero coefficients in the expansions in base $p_i$. Using the recursive relation of the sequence $(a^{i}_j(n))_{n \geq 0}$, according to Lemma~\ref{lem2} we have on the one hand

$$a^{i}_j(n+r)=a^{i}_j(n_i)+g^{i}_j(\beta'_{s_i-1,i},n_i)+\displaystyle\sum_{\nu=0}^{s_i-2}g^{i}_j(\beta'_{\nu,i},\beta'_{\nu+1,i}),$$
and on the other hand

$$a^{i}_j(n)=a^{i}_j(n_i)+g^{i}_j(\beta_{s_i-1,i},n_i)+\displaystyle\sum_{\nu=0}^{s_i-2}g^{i}_j(\beta_{\nu,i},\beta_{\nu+1,i}).$$

This implies that

\begin{align*}
&a^{i}_j(n+r)-a^{i}_j(n)\\
&=g^{i}_j(\beta'_{s_i-1,i},n_i)+\displaystyle\sum_{\nu=0}^{s_i-2}g^{i}_j(\beta'_{\nu,i},\beta'_{\nu+1,i})-g^{i}_j(\beta_{s_i-1,i},n_i)-\displaystyle\sum_{\nu=0}^{s_i-2}g^{i}_j(\beta_{\nu,i},\beta_{\nu+1,i}).
\end{align*}

Similarly, since $b_i+r=[\beta'_{v_i,i},\ldots,\beta'_{1,i},\beta'_{0,i}]_p$ and $b_i=[\beta_{w_i,i},\ldots,\beta_{1,i},\beta_{0,i}]_p$, and $\beta'_{v_i+1,i}=0$ and $\beta_{w_i+1,i}=0$, by definition of $v_i$ and $w_i$, we obtain 

$$a^{i}_j(b_i+r)=a^{i}_j(0)+\displaystyle\sum_{\nu=0}^{v_i}g^{i}_j(\beta'_{\nu,i},\beta'_{\nu+1,i})$$
and

$$a^{i}_j(b_i)=a^{i}_j(0)+\displaystyle\sum_{\nu=0}^{w_i}g^{i}_j(\beta_{\nu,i},\beta_{\nu+1,i}).$$

Consequently, we have

\begin{equation}
a^{i}_j(n+r)-a^{i}_j(n)=a^{i}_j(b_i+r)-a^{i}_j(b_i)+\mu_{i,j}(b_i,r,n_i)\label{relation}
\end{equation}
where

\begin{small}
\begin{align*}
&\mu_{i,j}(b_i,r,n_i)\\
&=g^{i}_j(\beta'_{s_i-1,i},n_i)-g^{i}_j(\beta_{s_i-1,i},n_i)+\displaystyle\sum_{\nu=v_i+1}^{s_i-2}g^{i}_j(\beta'_{\nu,i},\beta'_{\nu+1,i})-\displaystyle\sum_{\nu=w_i+1}^{s_i-2}g^{i}_j(\beta_{\nu,i},\beta_{\nu+1,i}).
\end{align*}
\end{small}

Moreover, we have $a(n+r)=a(n)$ if and only if $a^{i}_j(n+r)=a^{i}_j(n) \ \text{for all} \ 1 \leq i \leq d \ \text{and} \ 1 \leq j \leq k_i$. In what follows, we use the notation 
$$\textbf{a}=\textbf{a}(n)=\begin{pmatrix}
a^{1}_1(n+r)-a^{1}_1(n) \\
\vdots \\
a^{1}_{k_1}(n+r)-a^{1}_{k_1}(n) \\
\vdots \\
\vdots \\
a^{d}_1(n+r)-a^{d}_1(n) \\
\vdots \\
a^{d}_{k_d}(n+r)-a^{d}_{k_d}(n) \\
\end{pmatrix}$$
for the vector $a(n+r)-a(n)$. We also introduce the notation
$$\textbf{h}=\left(\cfrac{h^{1}_1}{p_1},\ldots,\cfrac{h^{1}_{k_1}}{p_1},\ldots \ldots,\cfrac{h^{d}_1}{p_d},\ldots,\cfrac{h^{d}_{k_d}}{p_d}\right).$$

Thus,

$$\displaystyle\sum_{n<N}\delta(n+r_1,n+r_2)=N\left(1-\cfrac{1}{k}\right)-\cfrac{1}{k}\displaystyle\sum_{n<N}\displaystyle\sum_{\textbf{h}\neq \textbf{0}}\text{e}(\textbf{h}\cdot \textbf{a}).$$

Fix a vector $\textbf{h}\neq \textbf{0}$ such that for all $1 \leq i \leq d$ and all $1 \leq j \leq k_i$ we have $0 \leq h^{i}_j < p_i$.

It suffices to estimate $\displaystyle\sum_{n<N}\text{e}(\textbf{h}\cdot \textbf{a})$. We define

$$\textbf{a'}=\begin{pmatrix}
a^{1}_1(n_1 {p_1}^{s_1}+b_1+r)-a^{1}_1(n_1 {p_1}^{s_1}+b_1) \\
\vdots \\
a^{1}_{k_1}(n_1 {p_1}^{s_1}+b_1+r)-a^{1}_{k_1}(n_1 {p_1}^{s_1}+b_1) \\
\vdots \\
\vdots \\
a^{d}_1(n_d {p_d}^{s_d}+b_d+r)-a^{d}_1(n_d {p_d}^{s_d}+b_d) \\
\vdots \\
a^{d}_{k_d}(n_d {p_d}^{s_d}+b_d+r)-a^{d}_{k_d}(n_d {p_d}^{s_d}+b_d) \\
\end{pmatrix},$$

$$\textbf{a''}=\begin{pmatrix}
a^{1}_1(b_1+r)-a^{1}_1(b_1) \\
\vdots \\
a^{1}_{k_1}(b_1+r)-a^{1}_{k_1}(b_1) \\
\vdots \\
\vdots \\
a^{d}_1(b_d+r)-a^{d}_1(b_d) \\
\vdots \\
a^{d}_{k_d}(b_d+r)-a^{d}_{k_d}(b_d) \\
\end{pmatrix}
\quad \text{and} \quad \boldsymbol{\mu}=\begin{pmatrix}
\mu_{1,1}(b_1,r,n_1) \\
\vdots \\
\mu_{1,k_1}(b_1,r,n_1) \\
\vdots \\
\vdots \\
\mu_{d,1}(b_d,r,n_d) \\
\vdots \\
\mu_{d,k_d}(b_d,r,n_d) \\
\end{pmatrix}
$$

Using \eqref{relation} we have

\begin{align}
\displaystyle\sum_{n<N}\text{e}(\textbf{h}\cdot \textbf{a})&=\displaystyle\sum_{\textbf{b} \in \mathcal{B}_0}\sum_{\substack{n<N \\ n \in P_{\textbf{b}}}}\text{e}(\textbf{h}\cdot \textbf{a'})+\displaystyle\sum_{\textbf{b} \in \mathcal{B} \setminus \mathcal{B}_0}\sum_{\substack{n<N \\ n \in P_{\textbf{b}}}}\text{e}(\textbf{h}\cdot \textbf{a'})\nonumber\\
&=\displaystyle\sum_{\textbf{b} \in \mathcal{B}_0}\sum_{\substack{n<N \\ n \in P_{\textbf{b}}}}\text{e}(\textbf{h}\cdot (\textbf{a''}+\boldsymbol{\mu}))+\displaystyle\sum_{\textbf{b} \in \mathcal{B} \setminus \mathcal{B}_0}\sum_{\substack{n<N \\ n \in P_{\textbf{b}}}}\text{e}(\textbf{h}\cdot \textbf{a'})\nonumber\\
&=\displaystyle\sum_{\textbf{b} \in \mathcal{B}}\text{e}(\textbf{h}\cdot \textbf{a''})\sum_{\substack{n<N \\ n \in P_{\textbf{b}}}}\text{e}(\textbf{h}\cdot \boldsymbol{\mu})\label{sum1}\\
&+\displaystyle\sum_{\textbf{b} \in \mathcal{B} \setminus \mathcal{B}_0}\sum_{\substack{n<N \\ n \in P_{\textbf{b}}}}(\text{e}(\textbf{h}\cdot \textbf{a'})-\text{e}(\textbf{h}\cdot (\textbf{a''}+\boldsymbol{\mu})))\label{sum2}.
\end{align}

Note that \begin{small}$\mathcal{B} \setminus \mathcal{B}_0=\left\{(b_1,\ldots,b_d):0\leq b_i < {p_i}^{s_i}, \ \exists j \in \{1,\ldots,d\}, \ b_j \geq {p_j}^{s_j}-r \right\}.$\end{small}

Then, $\mid \mathcal{B} \setminus \mathcal{B}_0 \mid \ll \displaystyle\sum_{i=1}^d \cfrac{r}{{p_i}^{s_i}}\prod_{j=1}^d {p_j^{s_j}}$.

Therefore, the sum~\eqref{sum2} is trivially bounded by

\begin{align*}
2\mid \mathcal{B} \setminus \mathcal{B}_0 \mid \#\left\{n<N : n \in P_{\textbf{b}}\right\}&\ll_k \left(\displaystyle\sum_{i=1}^d \cfrac{r}{{p_i}^{s_i}}\prod_{j=1}^d {p_j^{s_j}}\right) \left(\cfrac{N}{\prod_{i=1}^d{p_i}^{s_i}}+O(1)\right)\\
&\ll_k \ rN^{1-\frac{1}{d}}.
\end{align*}

We have one of the error terms in the estimate. Now, to finish the proof, we need to estimate \eqref{sum1}. Let

$$\mathcal{B}^r=\left\{\textbf{b} \in \mathcal{B} : v_i=w_i \ \text{and} \ \beta_{v_i,i}=\beta'_{w_i,i}, \ \text{for all} \ 1 \leq i \leq d  \right\}.$$

For every $\textbf{b} \in \mathcal{B}^r$ we have $\mu_{i,j}(b_i,r,n_i)=0$, for all $n < N, n \in P_{\textbf{b}}$. Using a similar argument where $\mathcal{B}^r$ corresponds to $\mathcal{B}_0$, we can bound the sum~\eqref{sum1} by

$$\ll \displaystyle\sum_{\textbf{b} \in \mathcal{B}}\text{e}(\textbf{h}\cdot \textbf{a''})\sum_{\substack{n<N \\ n \in P_{\textbf{b}}}}1+2\mid \mathcal{B} \setminus \mathcal{B}^r \mid\left(\cfrac{N}{\prod_{i=1}^d{p_i}^{s_i}}+O(1)\right).$$

The last part consists in establishing a bound for $\mid \mathcal{B} \setminus \mathcal{B}^r \mid$. Consider $t_i$ such that ${p_i}^{t_i}\leq r<{p_i}^{t_i+1}$. We have to count the number of $b_i$ satisfying $0\leq b_i<{p_i}^{s_i}$ and for which we have a carry propagation from digit $\beta_{v_i,i}$ of $b_i$ when adding $r$. For this, a necessary condition is

$$\beta_{t_i+1,i}=\beta_{t_i+2,i}=\cdots=\beta_{s_i-2,i}=p_i-1.$$

Then, $\mid \mathcal{B} \setminus \mathcal{B}^r \mid \leq \displaystyle\sum_{i=1}^d {p_i}^{t_i+1}+(s_i-2-t_i){p_i}^{t_i+2}$.

Using the fact that $s_i\leq\cfrac{\text{log}N^{\frac{1}{d}}}{\text{log}p_i}$, and $-t_i-1<-\cfrac{\text{log}r}{\text{log}p_i}$, we deduce

$$\mid \mathcal{B} \setminus \mathcal{B}^r \mid \leq \displaystyle\sum_{i=1}^d \left(rp_i+r{p_i}^2\left(\cfrac{\text{log}(N^{\frac{1}{d}})}{\text{log}p_i}-\cfrac{\text{log}r}{\text{log}p_i}\right)\right)\ll_k r\displaystyle\sum_{i=1}^d\text{log}N^{\frac{1}{d}}.$$

For all $1 \leq i \leq d$ and all $1 \leq j \leq k_i$, define $\textbf{h}^{i}=\left(\cfrac{h^{i}_1}{p_i},\ldots,\cfrac{h^{i}_{k_i}}{p_i}\right)$ and $\textbf{a}^{i}=(a^{i}_1(b_i+r)-a^{i}_1(b_i),\ldots,a^{i}_{k_i}(b_i+r)-a^{i}_{k_i}(b_i))$.

By adding all the terms, we have

\begin{align*}
\displaystyle\sum_{n<N}\text{e}(\textbf{h}\cdot \textbf{a})&=\displaystyle\sum_{\textbf{b} \in \mathcal{B}}\text{e}(\textbf{h}\cdot \textbf{a''})\sum_{\substack{n<N \\ n \in P_{\textbf{b}}}}1+O_k\left(rN^{1-\frac{1}{d}}+r\displaystyle\sum_{i=1}^d\text{log}N^{\frac{1}{d}}\right)\\
&=\left(\prod_{i=1}^d\displaystyle\sum_{b_i=0}^{{p_i}^{s_i}-1}\text{e}(\textbf{h}^{i} \cdot \textbf{a}^{i})\right)\left(\cfrac{N}{\prod_{i=1}^d{p_i}^{s_i}}+O(1)\right)+O_k\left(rN^{1-\frac{1}{d}}\right).
\end{align*}

By assumption, $\textbf{h}\neq \textbf{0}$ so there exists $1 \leq l \leq d$ such that $\textbf{h}^{l} \neq \textbf{0}$. With the notation of the proof of Theorem~\ref{theo4} we have 
$$\displaystyle\sum_{n<N^{1/d}}\text{e}(\textbf{h}^{l} \cdot \textbf{a}^{l})=S_{N^{1/d}}(h^{l}_1,\ldots,h^{l}_{k_l})=O_{p_l,k_l}\left(r\text{log}\frac{N^{\frac{1}{d}}}{r}+r\right).$$ 

For $i\neq l$, we bound the other factors trivially, and since $\forall i\in\{1,\ldots,d\}, \ {p_i}^{s_i}\leq N^{\frac{1}{d}}<{p_i}^{s_i+1}$, we obtain

\begin{align*}
\displaystyle\sum_{n<N}\text{e}(\textbf{h}\cdot \textbf{a})&\ll_k \left(N^{1-\frac{1}{d}}+N^{\frac{d-1}{d}}\right)\left(r\text{log}\frac{N^{\frac{1}{d}}}{r}+r\right)+rN^{1-\frac{1}{d}}\\
&\ll_k N^{\frac{d-1}{d}}\left(r\text{log}\frac{N^{\frac{1}{d}}}{r}+r\right).
\end{align*}

For $\textbf{h}\neq \textbf{0}$, we have $\underbrace{p_1\times \cdots\times p_1}_{k_1}\times\cdots\times\underbrace{p_d\times\cdots\times p_d}_{k_d}-1=k-1$ possible choices. Finally we have the estimate 

$$\displaystyle\sum_{n<N}\displaystyle\sum_{\textbf{h}\neq \textbf{0}}\text{e}(\textbf{h}\cdot \textbf{a})\  \ll_k (k-1)\left(N^{\frac{d-1}{d}}\left(r\text{log}\frac{N^{\frac{1}{d}}}{r}+r\right)\right),$$
where the implied constant only depends on $k$. This ends the proof of Theorem~\ref{theo5}.
\end{proof}

\section{Open questions}\label{Section6}

\begin{enumerate}
\item Is it possible to improve the error terms in Theorems~\ref{theo4} and~\ref{theo5}?
\item We have dealt with generalized Rudin--Shapiro sequences. Is it possible to obtain similar results for the discrete correlation of order $2$ for other constructions of pseudorandom sequences?
\item Our work concerns the discrete correlation of order $2$. What happens for correlations of higher order? As in Definition~\ref{def1}, it is possible to define the discrete correlation coefficient of order $m$ (see~\cite[p.346]{grant_bounds_2009}). For a uniform random sequence, Remark~\ref{rem1} still holds in this case, with $C_{\textbf{r}}=1-1/k^{m-1}$ with probability $1$ for all $m \geq 2$. So, a natural question arises: is it possible to build a family of pseudorandom sequences such that we obtain the expected main term for one or several $m \geq 3$ or for all $m \geq 2$? 
\end{enumerate}

\noindent {\bf Acknowledgements.} The author thanks I. Marcovici and T. Stoll for the supervision of this work and their useful advice. This work has been supported by the ANR Graal (ANR-14-CE25-0014), ANR-FWF Mudera (ANR-14-CE34-0009, FWF I-1751-N26) and the R\'egion Grand Est.

\bibliographystyle{amsplain}
\bibliography{Bibliocorr}

\end{document}